\documentclass[12pt,a4paper,leqno,verbatim]{amsart}
\newcounter{minutes}
\divide\time by 60
\newcounter{hours}
\multiply\time by 60 \addtocounter{minutes}{-\time}
\usepackage{amssymb}
\usepackage{hyperref}
\usepackage{graphicx}
\date{}
\newfont{\cyrilic}{wncyr10 scaled 1000}
\title{On the conjecture of generalized trigonometric and hyperbolic functions}
\author[B.A Bhayo]{Barkat Ali Bhayo}
\address{Department of Mathematical Information Technology, University of Jyv\"askyl\"a, 40014 Jyv\"askyl\"a, Finland}
\email{bhayo.barkat@gmail.com}

\author[L. Yin]{Li Yin}
\address[L. Yin]{Department of Mathematics and Information Science, Binzhou University, Binzhou City, Shandong Province, 256603, China}
\email{yinli\_79@163.com}

\newcommand{\comment}[1]{}

\swapnumbers
\theoremstyle{plain}

\newtheorem{theorem}[equation]{Theorem}
\newtheorem{lemma}[equation]{Lemma}

\newtheorem{corollary}[equation]{Corollary}

\numberwithin{equation}{section}

\begin{document}

%%%%%%KAUNIS K  \K %%%%%%%%%%%%%%
%\font\fFt=eusm10 %scaled 1200
%\font\fFa=eusm7  %scaled 1200
%\font\fFp=eusm5  %scaled 1200
%\def\K{\mathchoice
 %%%%displaystyle
%{\hbox{\,\fFt K}}
%%%%%textstyle
%{\hbox{\,\fFt K}}
%%%%scriptstyle
%{\hbox{\,\fFa K}}
%%%%%scriptscriptstyle
%{\hbox{\,\fFp K}}}
%%%%%%%%%%%%%
%\def\E{\mathchoice
 %%%%displaystyle
%{\hbox{\,\fFt E}}
%%%%%textstyle
%{\hbox{\,\fFt E}}
%%%%scriptstyle
%{\hbox{\,\fFa E}}
%%%%%scriptscriptstyle
%{\hbox{\,\fFp E}}}
%%%%%%%%%%%%%

\def\thefootnote{}
\footnotetext{ \texttt{\tiny File:~\jobname .tex,
          printed: \number\year-\number\month-\number\day,
          \thehours.\ifnum\theminutes<10{0}\fi\theminutes}
} \makeatletter\def\thefootnote{\@arabic\c@footnote}\makeatother

\thanks {The second author was supported by NSF of Shandong
Province under grant numbers ZR2012AQ028, and by the
Science Foundation of Binzhou University under grant BZXYL1303.}

\maketitle
%%%%%%%%%%%%%%%%%%%%%%%%%%%%%%%%%%%%%%%%%%%%%%%%%%%%%%%%%%%%%%%%%%%%%%
%%%%%%%%%%%%%%%%%%%%%%%%%%%%%%%%%%%%%%%%%%%%%%%%%%%%%%%%%%%%%%%%%%%%%%%
%%%%%%%%%%%%%%%%%%%%%%%%%%%%%%%%%%%%%%%%%%%%%%%%%%%%%%%%%%%%%%

\begin{abstract}
In this paper we prove the conjecture posed by Kl\'en et al. in \cite{kvz}, and give optimal inequalities for generalized
trigonometric and hyperbolic functions.
\end{abstract}



%%%%%%%%%%%%%%%%%%%%%%%%%%%%%%%%%%%%%%%%%%%%%%%%%%%%%%%%%%%%%%%%%%%%%%%%%%%%%%%%%%%%%%%%%%%%%%%%%
%%%%%%%%%%%%%%%%%%%%%%%%%%%%%%%%%  Introduction   %%%%%%%%%%%%%%%%%%%%%%%%%%%%%%%%%%%%%%%%%%%%%%%
%%%%%%%%%%%%%%%%%%%%%%%%%%%%%%%%%%%%%%%%%%%%%%%%%%%%%%%%%%%%%%%%%%%%%%%%%%%%%%%%%%%%%%%%%%%%%%%%%

\section{introduction}

In 1995, P. Lindqvist \cite{l} studied the generalized trigonometric
and hyperbolic functions with parameter $p>1$. Thereafter
several authors became interested to work on the equalities and inequalities of
these generalized functions, e.g, see
\cite{bv1,bv2,bv3,bbv,be,egl,jq,take} and the references therein.
Recently, Kl\'en et al. \cite{kvz} were motivated by many results on these generalized
trigonometric and hyperbolic functions, and they
generalized some classical inequalities in terms of generalized
trigonometric and hyperbolic functions, such as
Mitrinovi\'c-Adamovi\'c inequality, Huygens' inequality, and
Wilker's inequality.
In this paper we prove the conjecture posed by Kl\'en et al. in \cite{kvz}, and in Theorem \ref{thm3.2} we generalize the 
inequality
$$\frac{1}{\cosh(x)^a}<\frac{\sin(x)}{x}<\frac{1}{\cosh(x)^b},$$
where $a=\log(\pi/2)/\log(\cosh(\pi/2)) \approx 0.4909$ and $b=1/3$, due to Neuman and S\'andor
\cite[Theorem 2.1]{neusan}.
 
For the formulation of our main results we give the definitions of the
generalized trigonometric and hyperbolic functions as below.

%%%%%%%%%%%%%%%%%%%%%%%%%%%%%%%%%%%%%%%%%%%%%%
%In~\cite{kvz}, some classical inequalities for generalized
%trigonometric and hyperbolic functions, such as
%Mitrinovi\'c-Adamovi\'c inequality, Huygens' inequality, and
%Wilker's inequality were generalized. In~\cite{yhq}, some new second
%Wilker type inequalities for generalized trigonometric and
%hyperbolic functions were established. In~\cite{bbv}, some Tur\'an
%type inequalities for generalized trigonometric and hyperbolic
%functions were presented. Very recently, a conjecture posed
%in~\cite{bv2} was verified in~\cite{jq}.

The increasing homeomorphism function $F_{p}:[0,1]\to [0,\pi_{p}/2]$ is defined by
$$F_{p}(x)={\rm arcsin}_{p}(x)=\int_0^x{(1-t^p)}^{-1/p}\,dt,$$
and its inverse $\sin_{p,q}$ is
called generalized sine function, which is defined
on the interval $[0,\pi_{p}/2]$,
where
$${\rm arcsin}_{p}(1)=\pi_{p}/2.$$
The function $\sin_p$ is strictly increasing and concave on $[0,\pi_p/2]$, and it is also called the eigenfunction of the Dirichlet eigenvalue problem for the one-dimensional $p-$Laplacian \cite{dm}.
In the same way, we can define the
generalized cosine function, the generalized tangent, and
also the corresponding hyperbolic functions.

The generalized cosine function is defined by
$$\frac{d}{dx}\sin_{p}(x)=\cos_{p}(x),\quad x\in[0,\pi_{p}/2]\,.$$
It follows from the definition that
$$\cos_{p}(x)=(1-(\sin_{p}(x))^p)^{1/p}\,,$$
and
\begin{equation}\label{equ2}
|\cos_{p}(x)|^p+|\sin_{p}(x)|^p=1,\quad x\in\mathbb{R}.
\end{equation}
Clearly we get
$$\frac{d}{dx}\cos_p(x)=-\cos_p(x)^{2-p}\sin_p(x)^{p-1}.$$

The generalized tangent function $\tan_{p}$ is defined by
$$\tan_{p}(x)=\frac{\sin_{p}(x)}{\cos_{p}(x)}.$$

For $x\in(0,\infty)$, the inverse of generalized hyperbolic sine function $\sinh_{p}(x)$ is defined by
$${\rm arsinh}_{p}(x)=\int^x_0(1+t^p)^{-1/p}dt,\,$$
and generalized hyperbolic cosine and tangent functions are defined by
$$\cosh_{p}(x)=\frac{d}{dx}\sinh_{p}(x),\quad \tanh_{p}(x)=\frac{\sinh_{p}(x)}{\cosh_{p}(x)}\,,$$
respectively. It follows from the definitions that
\begin{equation}\label{equ3}
|\cosh_{p}(x)|^p-|\sinh_{p}(x)|^p=1.
\end{equation}
From above definition and (\ref{equ3})
we get the following derivative formulas,
$$\frac{d}{dx}\cosh_p(x)=\cos_p(x)^{2-p}\sin_p(x)^{p-1},\quad \frac{d}{dx}\tanh_p(x)=1-|\tanh_p(x)|^p.$$
Note that these generalized trigonometric and hyperbolic functions coincide with usual functions for $p=2$.

Our main result reads as follows:

\begin{theorem}\cite[Conjecture 3.12]{kvz}\label{thm3.1}
For $p\in[2,\infty)$, the function
$$f(x)=\frac{\log(x/\sin_p(x))}{\log(\sinh_p(x)/x)}$$
is strictly increasing from $(0,\pi _p/2)$ onto $(1,p)$. In particular,
$$\left(\frac{x}{\sinh_p(x)}\right)^p<\frac{\sin_p(x)}{x}<\frac{x}{\sinh_p(x)}.$$
\end{theorem}

\begin{theorem}\label{thm3.2}
For $p\in[2,\infty)$, the function
$$g(x)=\frac{\log(x/\sin_p(x))}{\log(\cosh_p(x))}$$
is strictly increasing in $x\in(0,\pi_p/2)$. In particular, we have
$$\frac{1}{\cosh_p(x)^\beta}<\frac{\sin_p(x)}{x}< \frac{1}{\cosh_p(x)^\alpha},$$
where
$\alpha  = 1/(1+p)$
and
$\beta=\log(\pi_p/2)/\log(\cosh_p(\pi_p/2))$
are the best possible constants.
\end{theorem}

%%%%%%%%%%%%%%%%%%%%%%%%%%%%%%%%%%%%%%%%%%%%%%%%%%%%%%%%%%%%%%%%%%%%%%%%%%%%%%%%%%%%%%%%%%
%%%%%%%%%%%%%%%%%%%%%%%%%%%%%% Preliminaries and proofs %%%%%%%%%%%%%%%%%%%%%%%%%%%%%%%%%%
%%%%%%%%%%%%%%%%%%%%%%%%%%%%%%%%%%%%%%%%%%%%%%%%%%%%%%%%%%%%%%%%%%%%%%%%%%%%%%%%%%%%%%%%%%
\section{Preliminaries and proofs}

The following lemmas will be used in the proof of main result.

\begin{lemma}\cite[Theorem 2]{avv1}\label{lem1}
For $-\infty<a<b<\infty$,
let $f,g:[a,b]\to \mathbb{R}$
be continuous on $[a,b]$, and be differentiable on
$(a,b)$. Let $g^{'}(x)\neq 0$
on $(a,b)$. If $f^{'}(x)/g^{'}(x)$ is increasing
(decreasing) on $(a,b)$, then so are
$$\frac{f(x)-f(a)}{g(x)-g(a)}\quad and \quad \frac{f(x)-f(b)}{g(x)-g(b)}.$$
If $f^{'}(x)/g^{'}(x)$ is strictly monotone,
then the monotonicity in the conclusion
is also strict.
\end{lemma}

\begin{lemma}\label{(2.1-eq)}
For $p\in[2,\infty)$, the function
$$f(x)=\frac{p\sin_p(x)\log\left(x/\sin_p(x)\right)}{\sin_p(x)-x\cos_p(x)}$$
is strictly decreasing from $(0,\pi_p/2)$ onto $(1,p\log(\pi_p/2))$. In particular,
$$\exp\left(\frac{1}{p}\left(\frac{x}{\tan_p(x)}-1\right)\right)
< \frac{\sin_p(x)}{x}<
\exp\left(\left(\log\frac{\pi_p}{2}\right)\left(\frac{x}{\tan_p(x)}-1\right)\right).$$
\end{lemma}

\begin{proof} Write
$$f_1(x)=p\sin_p(x)\log\left(x/\sin_p(x)\right), \quad f_2(x)=\sin_p(x)-x\cos_p(x),$$
and clearly $f_1(0)=f_2(0)=0$.
Differentiation with respect $x$ gives
\begin{eqnarray*}
\frac{f_1'(x)}{f_2'(x)}&=&\frac{(\sin_p(x))/x+\cos_p(x)(\log(x/\sin_p(x))-1)}{x\cos_p(x)^{2-p}\sin_p(x)^{p-1}}\\
&=&\frac{1}{x\tan_p(x)^{p-1}}\left(\frac{1}{x\cos_p(x)}+\log\left(\frac{x}{\sin_p(x)}\right)-1\right),
\end{eqnarray*}
which is the product of two decreasing functions, this implies that $f_1'/f_2'$ is decreasing.
Hence the function $f$ is decreasing by Lemma \ref{lem1}. The limiting values follows from the l'H\^{o}spital rule.
\end{proof}

\begin{lemma}\label{(2.2-eq)}
For $p\in[2,\infty)$ the function
$$g(x)=\frac{p\sinh_p(x)\log\left(\sinh_p(x)/x\right)}{x\cosh_p(x)-\sinh_p(x)}$$
is strictly increasing from $(0,\infty)$ onto $(1,p)$. In particular, we have
$$\exp\left(\frac{1}{p}\left(\frac{x}{\tanh_p(x)}-1\right)\right)
< \frac{\sinh_p(x)}{x}<
\exp\left(\left(\frac{x}{\tanh_p(x)}-1\right)\right).$$
\end{lemma}

\begin{proof}
Write
$$g_1(x)=\sinh_p(x)\log\left(\frac{\sinh_p(x)}{x}\right),\quad g_2(x)=x\cosh_p(x)-\sinh_p(x), $$
clearly $g_1(0)=g_2(0)=0$.
Differentiation with respect $x$ gives
\begin{eqnarray*}
\frac{g_1'(x)}{g_2'(x)}&=&\frac{\cosh_p(x)(1+\log(\sinh_p(x)/x)-\sinh_p(x)/x}{x\cosh_p(x)^{2-p}\sinh_p(x)^{p-1}}\\
&=& \frac{\sinh_p(x)}{x}\frac{\cosh_p(x)(1+\log(1+\sinh_p(x)/x)-\sinh_p(x)/x}{\cosh_p(x)\tanh_p(x)^p},
\end{eqnarray*}
which is increasing, this implies that $g$ is increasing. The limiting values follows from the l'H\^{o}spital rule.
\end{proof}

%%%%%%%%%%%%%%%%%%%%%%%%%

%\section{lemmas}
%\begin{lemma}\label{2.1-lem}
%For $p\geq2$ and $x\in(0,\frac{\pi _p}{2})$, then
%\begin{equation}\label{(2.1-eq)}
%\ln \bigl(\frac{x}{{\sin _p x}}\bigr) < \frac{{\sin _p x - x\cos _p x}}{{p\sin _p x}}
%\end{equation}
%\end{lemma}
%
%\begin{lemma}\label{2.2-lem}
%For $p\geq2$ and $x\in(0,\frac{\pi _p}{2})$, then
%\begin{equation}\label{(2.2-eq)}
%\ln \bigl(\frac{{\sinh _p x}}{x}\bigr) > \frac{{x\cosh _p x - \sinh _p x}}{{p\sinh _p x}}.
%\end{equation}
%\end{lemma}

\begin{lemma}\label{2.3-lem}
For all $x>0$ and $p>1$, we have
$$
\log(\cosh _p (x)) > \frac{x} {p}\tanh _p(x)^{p - 1} x.
$$
\end{lemma}
\begin{proof}
Let \begin{equation*}f(x) = \log (\cosh _p (x)) - \frac{x} {p}\tanh _p(x)^{p
- 1}.
\end{equation*}

A simple computation yields
\begin{eqnarray*}
  f'(x) &=& \tanh _p(x)^{p - 1}  - \left( \frac{\tanh _p(x)^{p - 1} }
{p} + \frac{(p - 1)}{p}\frac{x\tanh _p(x)^{p - 2}}{\cosh _p(x)^p } \right) \\
   &=& \frac{p - 1}{p}\tanh _p(x)^{p - 2} \left(\tanh _p(x) - \frac{x}{\cosh _p(x)^p } \right),
\end{eqnarray*}
which is positive because $ \sinh _p (x) > x $ and $ \cosh _p (x) > 1 $
for all $x>0$. Thus $f(x)$ is strictly increasing and
$f(x)>f(0)=0$, this implies the proof.
\end{proof}

\noindent{\bf Proof of Theorem \ref{thm3.1}.} \rm  Write $f(x)=f_1(x)/f_2(x)$ for $x\in(0,\pi_p/2)$, where
$$f_1 (x) = \log \left( {\frac{x}{{\sin _p (x)}}} \right),f_2 (x) = \log \left( {\frac{{\sinh _p (x)}}{x}} \right).$$
For the proof of the monotonicity of the function $f$, it is enough to prove that
$$f'(x)=\frac{f_1'(x)f_2(x)-f_1(x)f_2'(x)}{f_2(x)^2}$$
is positive. After simple computation, this is equivalent to write 
$$
 x\left( {f_2 (x)} \right)^2 f'(x) = \displaystyle\frac{{\sin _p (x) - x\cos _p (x)}}{{\sin _p (x)}}f_2(x)
  - \displaystyle\frac{{x\cosh _p (x) - \sinh _p (x)}}{{\sinh _p (x)}}f_1(x),
 $$
which is positive by Lemmas \ref{(2.1-eq)} and \ref{(2.2-eq)}. Hence, $f$ is strictly increasing, and limiting values 
follows by applying the l'H\^opital rule.
This completes the proof.$\hfill\square$

\bigskip

\noindent{\bf Proof of Theorem \ref{thm3.2}.} \rm Write $g(x)=g_1(x)/g_2(x)$ for $x\in(0,\pi_p/2)$, where
$g_1 (x) = \log(x/\sin_p(x)),\, g_2 (x) = \log(\cosh_p(x))$.
Here we give the same argument as in the proof of the Theorem \ref{thm3.1}, and compute similary
\begin{eqnarray*}
\left( \log \left( \cosh _p (x) \right) \right)^2 g'(x)& =& \frac{\sin _p (x) - x\cos _p (x)}{x\sin _p (x)}
\log  \cosh _p (x) - \tanh _p(x)^{p - 1} \log \left( \frac{x}{\sin _p (x)} \right) \\
 &>& \frac{\sin _p (x) - x\cos _p (x)}{x\sin _p (x)\tanh _p(x)^{1-p} }\frac{x}{p}-\frac{\sin _p (x) - x\cos _p (x)}
{p\sin _p (x)\tanh _p(x)^{1-p} }\\
&=&0,
\end{eqnarray*}
by Lemmas \ref{(2.1-eq)} and \ref{2.3-lem}.
The limiting values follow from the l'H\^{o}spital rule easily, hence the proof is obvious.
$\hfill\square$\\

The following corollary follows from \cite[Lemma 3.3]{kvz} and Theorem \ref{thm3.2}.

\begin{corollary} For $p\in[2,\infty)$ and $x\in(0,\pi_p/2)$, we have
$$\cos_p(x)^\beta<\frac{1}{\cosh_p(x)^\beta}<\frac{\sin_p(x)}{x}< \frac{1}{\cosh_p(x)^\alpha}<1,$$
where $\alpha$ and $\beta$ are as in Theorem \ref{thm3.2}.
\end{corollary}

\vspace{.5cm}

{\sc Acknowledgments.} The authors are indebted to the 
anonymous referee for his/her valuable comments.

%%%%%%%%%%%%%%%%%%%%%%%%%%%%%%%%%%%%%%%%%%%%%%%%%%%%%%%%%%%%%%%%%%%%%%%%%%%%%%%%%%%%%%%%%%%%%%
%%%%%%%%%%%%%%%%%%%%%%%%%%%%%%%%%%%%%% References %%%%%%%%%%%%%%%%%%%%%%%%%%%%%%%%%%%%%%%%%%%%
%%%%%%%%%%%%%%%%%%%%%%%%%%%%%%%%%%%%%%%%%%%%%%%%%%%%%%%%%%%%%%%%%%%%%%%%%%%%%%%%%%%%%%%%%%%%%%

\end{document}